\documentclass[12pt,titlepage]{article} 
\usepackage{amsmath}                                                                         
\usepackage{amsfonts,amssymb,amsthm}                                                         
\usepackage[all,knot,poly]{xy}
\usepackage{graphicx}

\newtheorem{theorem}{Theorem}[section]
\newtheorem{lemma}[theorem]{Lemma}          
\newtheorem{proposition}[theorem]{Proposition}

\theoremstyle{definition}
\newtheorem{definition}[theorem]{Definition}
\newtheorem{example}[theorem]{Example}

\theoremstyle{remark}

\numberwithin{equation}{section}     
\allowdisplaybreaks[1]                                                      

\newcommand{\bC}{{\mathbf C}}

\newcommand{\bZ}{{\mathbf Z}}

\DeclareMathOperator{\Tr}{Tr}
\providecommand{\abs}[1]{\lvert#1\rvert}
\providecommand{\Babs}[1]{\Bigl\lvert#1\Bigr\rvert}

\newcommand{\ra}{\rightarrow}                                                                
                                                           
\begin{document}                                                                             
\title{The Jones Polynomial and Khovanov Homology of Weaving Knots $W(3,n)$}
\author{
Rama Mishra
\thanks{We thank the office of the Dean of the College of Arts and Sciences,
  New Mexico State University, for arranging a visiting appointment.}
\\
Department of Mathematics
\\
IISER Pune
\\
Pune, India
\and 
Ross Staffeldt 
\thanks{We thank the office of the Vice-President for Research, New Mexico
  State University, for arranging a visit to IISER-Pune to discuss
  implementation of an MOU between IISER-Pune and NMSU.}
\\Department of Mathematical Sciences\\
New Mexico State University\\
Las Cruces, NM 88003 USA}
\maketitle
\begin{abstract}
In this paper we compute the signature for a family of knots $W(k,n)$, the weaving knots of type $(k,n)$. 
By work of E.~S.~Lee the signature calculation implies a vanishing theorem for the Khovanov homology of weaving knots. 
Specializing to knots $W(3,n)$, we develop recursion relations that enable us to compute the Jones polynomial of $W(3,n)$.
Using additional results of Lee, we compute the ranks of the Khovanov Homology of these knots.
 At the end we provide evidence for our conjecture that, asymptotically, 
 the ranks of Khovanov Homology of $W(3,n)$ are {\it normally distributed}.
\end{abstract}
\section{Introduction} \label{Introduction}
Weaving knots originally attracted interest, because it was conjectured that their complements would have
the largest hyperbolic volume for a fixed crossing number.  
Here is the weaving knot $W(3,4)$.
\begin{figure}[h!]  
  \begin{equation*}
  \xygraph{  !{0;/r1.5pc/:}
!{\hcap}[u]
!{\hcap[3]}[u]
!{\hcap[5]}[llllllll]
!{\xcaph[-8]@(0)}[dl]
!{\xcaph[-8]@(0)}[dl]
!{\xcaph[-8]@(0)}[uul]
!{\hcap[-5]}[d]
!{\hcap[-3]}[d]
!{\hcap[-1]}[d]
!{\xcaph[1]@(0)}[dl]
!{\htwist}[d]
!{\xcaph[1]@(0)}[uul]
!{\htwistneg}
!{\xcaph[1]@(0)}[dl]
!{\htwist}[d]
!{\xcaph[1]@(0)}[uul]
!{\htwistneg}
!{\xcaph[1]@(0)}[dl]
!{\htwist}[d]
!{\xcaph[1]@(0)}[uul]
!{\htwistneg}
!{\xcaph[1]@(0)}[dl]
!{\htwist}[d]
!{\xcaph[1]@(0)}[uul]
!{\htwistneg}
}    
  \end{equation*}
\end{figure} \vspace{-8em}
\newline 
Enumerating strands $1, \ldots, p$ from the outside inward, our example is
the closure of the braid $(\sigma_1 \sigma_2^{-1})^4$ on three strands.
Thus, $\sigma_1$ is a righthand twist involving strands 1 and 2, and 
$\sigma_2$ is a righthand twist involving strands 2 and 3, and so on.

In words, the weaving knot $W(p,q)$ is obtained from the torus knot $T(p,q)$
by making the standard diagram of the torus knot alternating. Symbolically, $T(p,q)$ is the
closure of the braid $(\sigma_1 \sigma_2 \cdots \sigma_{p-1})^q$, and
$W(p,q)$ is the closure of the braid $(\sigma_1 \sigma_2^{-1} \cdots \sigma_{p-1}^{\pm 1})^q$.
Obviously, the parity of $p$ is important.  If the greatest common divisor $\gcd(p,q) > 1$, 
then $T(p,q)$ and $W(p,q)$ are both links with $\gcd(p,q)$ components.  In general we
are interested only in the cases when $W(p,q)$ is an actual knot.

In \cite{Weaving_vol}  the main result is the following theorem.  
\begin{theorem}
  [Theorem 1.1, \cite{Weaving_vol}] 
If $p \geq 3$ and $q \geq 7$, then
\begin{equation*}
  v_{{\rm oct}}(p-2)\,q\,\biggl(1 - \frac{(2\pi)^2}{q^2}\biggr)^{3/2} \leq {\rm vol}(W(p,q)) 
             < \bigl(v_{{\rm oct}}(p-3) + 4\,v_{{\rm tet}}) q.
\end{equation*}
\end{theorem}
Champanerkar, Kofman, and Purcell call these bounds asymptotically sharp because their ratio approaches 1, 
as $p$ and $q$ tend to infinity.  Since the crossing number of $W(p,q)$ is known to be $(p{-}1)q$, the volume 
bounds in the theorem imply
\begin{equation*}
  \lim_{p,q \to \infty}\frac{{\rm vol}(W(p,q))}{c(W(p,q))} = v_{{\rm oct}} \approx 3.66
\end{equation*}
Their study raised the question of examining the asymptotic behavior of  other invariants of weaving knots.
In this paper we start a study of the asymptotic behavior of Khovanov homology of weaving knots.

Briefly, since weaving knots are alternating knots by definition, we may specialize certain properties of the Khovanov homology of
alternating knots to get started.  This is accomplished in section \ref{Weaving}, where we explain how to calculate the signature of
weaving knots. The second main ingredient in our analysis is the fact that for alternating knots knowing the Jones polynomial is 
equivalent to knowing the Khovanov homology.  How this works explicitly in our examples is explained in section \ref{Jones-to-Khovanov}.
In section \ref{Hecke} we prepare to  follow the development of the Jones polynomial in \cite{Jones_poly86}, starting from 
representations of braid groups into Hecke algebras.  For weaving knots $W(3,n)$, which are naturally represented as the closures 
of braids on three strands, we develop recursive formulas for their representations in the Hecke algebras. These formulas are used 
in computer calculations of the Jones polynomials we need.  Section \ref{JonesPoly} builds on the recursion formulas 
to develop information about the Jones polynomials $V_{W(3,n)}(t)$.
After we explain how to obtain the two-variable Poincar\'{e} polynomial for Khovanov homology in section \ref{Jones-to-Khovanov},
we present the results of calculations in a few relatively small examples.  Our observation is that the distributions of 
dimensions in Khovanov homology resemble normal distributions.  We explore this further in section \ref{Data}, where we present
tables displaying summaries of calculations for weaving knots $W(3,n)$ for selected values of $n$ satisfying $\gcd(3, n) = 1$ 
and ranging up to $n=326$.  The standard deviation $\sigma$ of the normal distribution we attach to the Khovanov homology
of a weaving knot is a significant parameter.  The geometric significance of this number is an open question. 
\section{Generalities on Weaving knots} \label{Weaving}
We have already mentioned that weaving knots are alternating by definition.
Various facts about alternating knots facilitate our calculations of the Khovanov homology of weaving knots $W(3,n)$.
For example, we appeal first to the following theorem of Lee.
\begin{theorem}[Theorem 1.2, \cite{Lee_Endo04}]  \label{locateKHLee}
For any alternating knot $L$ the Khovanov invariants
${\mathcal H}^{i,j}(L)$ 
are supported in two lines
\begin{equation*}
  j = 2i -\sigma(L) \pm 1.  \qed
\end{equation*}
 \end{theorem}
We will see that this result also has several practical implications.   For example, to obtain a vanishing result for a particular alternating
knot, it suffices to compute the signature.  Indeed, it turns out that there is a combinatorial formula for the
signature of oriented non-split alternating links. To state the formula requires the following terminology.
\begin{definition} \label{crossings}
  For a link diagram $D$ let $c(D)$ be the number of crossings of $D$, let $x(D)$ be number of
negative crossings, and let $y(D)$ be the number of positive crossings. For an
oriented link diagram, let $o(D)$ be the number
of components of $D(\emptyset)$, the diagram obtained by $A$-smoothing every crossing.
\end{definition}
\vspace{-3em}
\begin{figure}[h!]
  \begin{minipage}[h!]{0.5\linewidth} 
  \centering
\begin{equation*}
\UseComputerModernTips  \xygraph{ !{0;/r3pc/:} 
!{\htwist=>}[rr] 
!{\htwistneg=>}[rr]
}
\end{equation*}   
\caption{Positive and negative crossings}
  \end{minipage}
 \begin{minipage}[h!]{0.5\linewidth}
  \centering
\begin{equation*}
\UseComputerModernTips  \xygraph{ !{0;/r3pc/:} 
!{\huntwist}[rr]
!{\vuntwist}
}
\end{equation*}   
\caption{$A$-smoothing a positive, resp., negative, crossing}
  \end{minipage}
   \label{fig:crossings_smoothings}
\end{figure}
In words, $A$-regions in a neighborhood of a crossing are the regions swept out as the upper strand sweeps
counter-clockwise toward the lower strand.  An $A$-smoothing removes the crossing to connect these regions.
With these definitions,  we may cite the following proposition.
\begin{proposition}[Proposition 3.11, \cite{Lee_Endo04}]  
\label{basic_signature} 
For an oriented non-split alternating link $L$ and a reduced alternating diagram $D$ of $L$, 
$\sigma(L) = o(D) - y(D) -1$. \qed
\end{proposition}
We now use this result to compute the signatures of weaving knots. 
For a knot or link $W(m, n)$ drawn in the usual way,  the number of crossings $c(D) = (m{-}1)n$.
In particular, for $W(2k{+}1,n)$, $c\bigl(W(2k{+}1,n)\bigr)= 2kn$;   for $W(2k, n)$, $c\bigl( W(2k, n) \bigr) = (2k{-}1)n$. 
Evaluating the other quantities in definition \ref{crossings}, we calculate the signatures of weaving knots. 
\begin{proposition} \label{weavingsignature}
  For a weaving knot $W(2k{+}1,n)$, $\sigma\bigl( W(2k{+}1,n) \bigr) = 0$,  
and for $W(2k, n)$,   $\sigma\bigl( W(2k, n) \bigr) =  -n{+}1$.
\end{proposition}
\begin{proof}
  Consider first the example $W(3,n)$, illustrated by figures \ref{fig:w34} and \ref{fig:w34a} for $W(3,4)$ drawn below.
After $A$-smoothing the diagram, the outer ring of crossings produces a circle bounding the 
rest of the smoothed diagram.
On the inner ring of crossings the $A$-smoothings produce $n$ circles in a cyclic arrangement.  Therefore
$o\bigl( W(3,n) \bigr) = 1 + n$.  The outer ring of crossings consists of positive crossings
and the inner ring of crossings consists of negative crossings, so $x(D) = y(D) = n$.  
Applying the formula of theorem \ref{basic_signature}, we obtain the result $\sigma\bigl( W(3,n) \bigr)=0$.  

For the general case $W(2k{+}1, n)$, we have the following considerations.
The crossings are organized into  $2k$ rings.  Reading from the outside toward the center, we have
first a ring of positive crossings, then a ring of negative crossings, and so on, alternating
positive and negative.  Thus $y(D) = kn$.  Considering the $A$-smoothing of the diagram of $W(2k{+}1,n)$,
as in the special case, a bounding circle appears from the smoothing of the outer ring.  
A chain of $n$ disjoint smaller circles appears inside the second ring.  No circles appear in the third ring, nor in any odd-numbered ring 
thereafter.  On the other hand, chains of $n$ disjoint smaller circles appear in each even-numbered ring.  Since there are $k$ even-numbered
rings, we have $o(D) = 1 + kn$. 
Applying the formula of proposition \ref{basic_signature}
\begin{equation*}
  \sigma\bigl( W(2k{+}1, n)\bigr)  = o(D) - y(D) -1 = (1+kn) - kn -1 = 0.
\end{equation*}
These figures illustrate the main points of the $W(3,n)$-cases, and, as explained above, the main points of the $W(2k{+}1, n)$-cases.
\begin{figure}[h!]
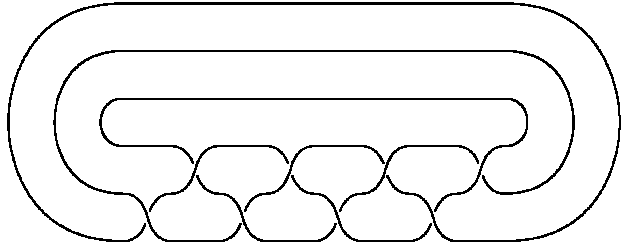

  \centering
\begin{equation*}
\xygraph{  !{0;/r1.5pc/:}
!{\hcap}[u]
!{\hcap[3]}[u]
!{\hcap[5]}[llllllll]
!{\xcaph[-8]@(0)}[dl]
!{\xcaph[-8]@(0)}[dl]
!{\xcaph[-8]@(0)}[uul]
!{\hcap[-5]}[d]
!{\hcap[-3]}[d]
!{\hcap[-1]}[d]
!{\xcaph[1]@(0)}[dl]
!{\htwist}[d]
!{\xcaph[1]@(0)}[uul]
!{\htwistneg}
!{\xcaph[1]@(0)}[dl]
!{\htwist}[d]
!{\xcaph[1]@(0)}[uul]
!{\htwistneg}
!{\xcaph[1]@(0)}[dl]
!{\htwist}[d]
!{\xcaph[1]@(0)}[uul]
!{\htwistneg}
!{\xcaph[1]@(0)}[dl]
!{\htwist}[d]
!{\xcaph[1]@(0)}[uul]
!{\htwistneg}
}   
\end{equation*}   \vspace{-9em}
  \caption{The weaving knot $W(3,4)$}
  \label{fig:w34}
\end{figure}
\begin{figure}[h!]
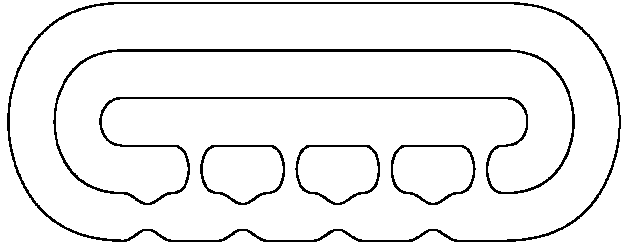

  \centering
\begin{equation*}  
  \xygraph{  !{0;/r1.5pc/:}
!{\hcap}[u]
!{\hcap[3]}[u]
!{\hcap[5]}[llllllll]
!{\xcaph[-8]@(0)}[dl]
!{\xcaph[-8]@(0)}[dl]
!{\xcaph[-8]@(0)}[uul]
!{\hcap[-5]}[d]
!{\hcap[-3]}[d]
!{\hcap[-1]}[d]
!{\xcaph[1]@(0)}[dl]
!{\huntwist}[u]
!{\huncross}[ddl]
!{\xcaph[1]@(0)}[uu]
!{\xcaph[1]@(0)}[dl]
!{\huntwist}[u]
!{\huncross}[ddl]
!{\xcaph[1]@(0)}[uu]
!{\xcaph[1]@(0)}[dl]
!{\huntwist}[u]
!{\huncross}[ddl]
!{\xcaph[1]@(0)}[uu]
!{\xcaph[1]@(0)}[dl]
!{\huntwist}[u]
!{\huncross}[ddl]
!{\xcaph[1]@(0)}[uu]
}
\end{equation*}  \vspace{-9em}
  \caption{The $A$-smoothing of $W(3,4)$}
  \label{fig:w34a}
\end{figure}

For the case $W(2k, n)$,  we show  $W(4,5)$ below in figures \ref{fig:w45} and \ref{fig:w45a} as an example. 
Our standard diagram may be organized into $2k{-}1$ rings of crossings.  
In each ring there are $n$ crossings, so the total number of crossings is $c(D) = (2k{-}1)n$. 
In our standard representation, there is an outer ring of $n$ positive crossings, next a ring of $n$ negative crossings, alternating 
until we end with an innermost ring of $n$ positive crossings.  There are thus $k$ rings of $n$ positive crossings and $k{-}1$ rings
of $n$ negative crossings.  Therefore,  $y(D) = kn$  and  $x(D) = (k{-}1)n$.
 Considering the $A$-smoothing of the diagram, a bounding circle appears from the smoothing of the outer ring.  As before, a chain of $n$ disjoint
smaller circles appears inside the second ring and in each successive even-numbered ring.  As previously noted, there are $k{-}1$ of these
rings.  No circles appear in odd-numbered rings, until we reach the last ring, where an inner bounding circle appears. 
Thus, $o(D) = 1 + (k{-}1)n + 1 = (k{-}1)n + 2$. 
Consequently, 
\begin{equation*}
    \sigma\bigl( W(2k, n)\bigr)  = o(D) - y(D) -1 = \bigl((k{-}1)n + 2\bigr) - kn - 1
                    = -n{+}1.  \qedhere
\end{equation*}
\end{proof}
\begin{figure}[h!]
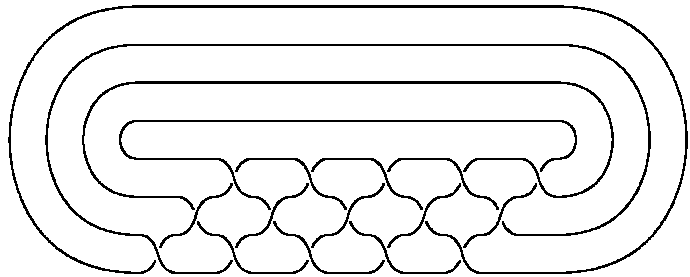

  \centering
\begin{equation*}
   \xygraph{  !{0;/r1.2pc/:}
!{\hcap}[u]
!{\hcap[3]}[u]
!{\hcap[5]}[u]
!{\hcap[7]}[lllllllllll]
!{\xcaph[-11]@(0)}[dl]
!{\xcaph[-11]@(0)}[dl]
!{\xcaph[-11]@(0)}[dl]
!{\xcaph[-11]@(0)}[uuul]
!{\hcap[-7]}[d]
!{\hcap[-5]}[d]
!{\hcap[-3]}[d]
!{\hcap[-1]}[d]
!{\xcaph[2]@(0)}[dl]
!{\xcaph[1]@(0)}[dl]
!{\htwist}[d]
!{\xcaph[1]@(0)}[uul]
!{\htwistneg}[u]
!{\htwist}[ddl]
!{\htwist}[d]
!{\xcaph[1]@(0)}[uul]
!{\htwistneg}[ul]
!{\xcaph[1]@(0)}
!{\htwist}[ddl]
!{\htwist}[d]
!{\xcaph[1]@(0)}[uul]
!{\htwistneg}[ul]
!{\xcaph[1]@(0)}
!{\htwist}[ddl]
!{\htwist}[d]
!{\xcaph[1]@(0)}[uul]
!{\htwistneg}[ul]
!{\xcaph[1]@(0)}
!{\htwist}[ddl]
!{\htwist}[d]
!{\xcaph[2]@(0)}[uul]
!{\htwistneg}[ul]
!{\xcaph[1]@(0)}
!{\htwist}[ddl]
!{\xcaph[1]@(0)}
}
\end{equation*} \vspace{-9em}
  \caption{The weaving knot $W(4,5)$}
  \label{fig:w45}
\end{figure}
\begin{figure}[h!]
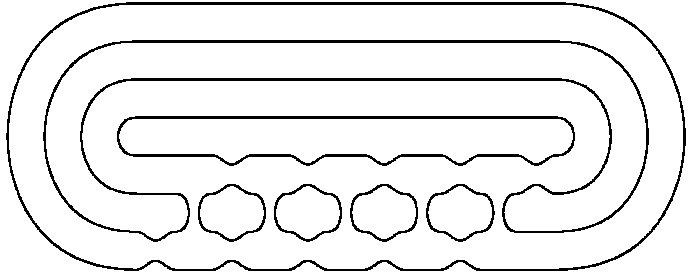

  \centering
\begin{equation*}  
  \xygraph{ !{0;/r1.2pc/:} 
!{\hcap}[u]
!{\hcap[3]}[u]
!{\hcap[5]}[u]
!{\hcap[7]}[lllllllllll]
!{\xcaph[-11]@(0)}[dl]
!{\xcaph[-11]@(0)}[dl]
!{\xcaph[-11]@(0)}[dl]
!{\xcaph[-11]@(0)}[uuul]
!{\hcap[-7]}[d]
!{\hcap[-5]}[d]
!{\hcap[-3]}[d]
!{\hcap[-1]}[d]
!{\xcaph[2]@(0)}[dl]
!{\xcaph[1]@(0)}[dl]
!{\huntwist}[d]
!{\xcaph[1]@(0)}[uul]
!{\huncross}[u]
!{\huntwist}[ddl]
!{\huntwist}[d]
!{\xcaph[1]@(0)}[uul]
!{\huncross}[ul]
!{\xcaph[1]@(0)}
!{\huntwist}[ddl]
!{\huntwist}[d]
!{\xcaph[1]@(0)}[uul]
!{\huncross}[ul]
!{\xcaph[1]@(0)}
!{\huntwist}[ddl]
!{\huntwist}[d]
!{\xcaph[1]@(0)}[uul]
!{\huncross}[ul]
!{\xcaph[1]@(0)}
!{\huntwist}[ddl]
!{\huntwist}[d]
!{\xcaph[2]@(0)}[uul]
!{\huncross}[ul]
!{\xcaph[1]@(0)}
!{\huntwist}[ddl]
!{\xcaph[1]@(0)}
}
\end{equation*} \vspace{-9em}  
  \caption{The $A$-smoothing of $W(4,5)$}
  \label{fig:w45a}
\end{figure}
\begin{theorem}
  \label{locateKH}
For a weaving knot $W(2k{+}1,n)$ the non-vanishing Khovanov homology ${\mathcal H}^{i,j}\bigl( W(2k{+}1, n) \bigr)$ lies on the 
lines
\begin{equation*}
  j = 2i \pm 1.
\end{equation*}
For a weaving knot $W(2k, n)$ the non-vanishing Khovanov homology ${\mathcal H}^{i,j}\bigl( W(2k, n) \bigr)$ lies on the lines
\begin{equation*}
  j = 2i + n -1 \pm 1
\end{equation*}
\end{theorem}
\begin{proof}
  Substitute the calculations made in lemma \ref{weavingsignature} into
  the formula of theorem \ref{locateKHLee}.
\end{proof}
\section{Recursion in the Hecke algebra} \label{Hecke}
We review briefly the definition of the Hecke algebra $H_{N+1}$ on generators $T_1$ through $T_N$,
and we define the representation of the braid group $B_3$ on three strands in $H_3$. 
Theorem \ref{heckerecursion} sets up recursion relations for the coefficients in the expansion 
of the image in $H_3$ of the braid $(\sigma_1 \sigma_2^{-1})^{n}$,
whose closure is the weaving knot $W(3,n)$.  The recursion relations are  essential for automating 
the calculation of the Jones polynomial for the knots $W(3,n)$.  
Proposition \ref{C121} uses these relations developed in theorem \ref{heckerecursion}
to prove a vanishing result for one of the coefficients.  Being able to ignore one of 
the coefficients speeds up the computations slightly.
\begin{definition}
  \label{Heckealgebras}
Working  over the ground field $K$ containing an element $q \neq 0$, the Hecke
algebra $H_{N+1}$ is the associative algebra with $1$ on generators $T_1$, \ldots, $T_N$ satisfying these relations.
\begin{align}
  T_iT_j &= T_jT_i, \quad \text{whenever $\abs{i-j} \geq 2$,}  \label{commutativity}
\\
  T_iT_{i+1}T_i &= T_{i+1}T_iT_{i+1}, \quad \text{for $1 \leq i \leq N{-}1$,} \label{interchange} 
\intertext{and,  finally,}
  T_i^2 &= (q{-}1)T_i + q,\quad  \text{for all $i$.} \label{inverse}
\end{align}
It is well-known 
\cite{Jones_poly86}
that  $(N{+}1)!$ is the dimension of $H_{N+1}$ over $K$.
\end{definition}
Recasting the relation 
$T_i^2 = (q{-}1)T_i + q$ in the form $q^{-1} \bigl(T_i - (q-1)\bigr)\cdot T_i = 1$ shows that 
$T_i$ is invertible in $H_{N+1}$ with 
$T_i^{-1} = q^{-1}\bigl( T_i - (q-1) \bigr)$.
Consequently, the specification
$\rho(\sigma_i)  = T_i$, combined with relations \eqref{commutativity} and \eqref{interchange}, defines  a homomorphism 
$\rho \colon B_{N+1} \ra H_{N+1} $ from $B_{N+1}$, the group of braids on $N{+}1$ strands,
 into the multiplicative monoid of $H_{N+1}$. 

For work in $H_3$, choose the  ordered basis
$\{1, T_1, T_2, T_1T_2, T_2T_1, T_1T_2T_1\}$.
The word in the Hecke algebra corresponding to the knot $W(3,n)$ is formally
\begin{equation}
  \label{eq:basicw3n}
 \rho\bigl( (T_1T_2^{-1})^n\bigr)  
        = q^{-n}\bigl(C_{n,0}+ C_{n,1}\cdot T_1 + C_{n,2} \cdot T_2 + C_{n,12} \cdot T_1T_2  
               + C_{n,21} \cdot T_2T_1  + C_{n,121} \cdot T_1T_2T_1 \bigr),
\end{equation}
where the coefficients  $C_{n,*} = C_{n,*}(q)$ of the monomials in $T_1$ and $T_2$ are  polynomials in $q$.
For $n = 1$, 
  \begin{equation*}
  \rho ( \sigma_1 \sigma_2^{-1}) = T_1T_2^{-1} = q^{-1}\cdot\bigl(  T_1 ( -(q{-}1) + T_2)  \bigr) 
  = q^{-1}\bigl( -(q{-}1)\cdot T_1 + T_1T_2 \bigr),
  \end{equation*}
so we have initial values 
\begin{equation}  \label{initialCs}
  C_{1,0}(q) = 0, \; C_{1,1}(q) = -(q{-}1), \; C_{1,2}(q) = 0, \; C_{1,12}(q) = 1,
                   \;  C_{1,21}(q) = 0, \; \text{and} \; C_{1,121}(q) = 0.
\end{equation}
\begin{theorem}
  \label{heckerecursion}
These polynomials satisfy the following recursion relations.
\begin{align}
C_{n,0}(q) &= q^2\cdot C_{n-1,21}(q) - q(q{-}1)\cdot C_{n-1,1}(q)  \label{cn0}
\\
  C_{n,1}(q) &= - (q{-}1)^2\cdot C_{n-1,1}(q) - (q{-}1)\cdot C_{n-1,0}(q) + q^2\cdot C_{n-1,121}(q) \label{cn1}
\\
  C_{n,2}(q) &=  q\cdot C_{n-1,1}(q)  \label{cn2}
\\
 C_{n,12}(q)  &= (q{-}1)\cdot C_{n-1,1}(q) + C_{n-1,0}(q)  \label{cn12}
\\
\begin{split}
  C_{n, 21}(q) &= -(q{-}1)\cdot C_{(n-1),2}(q) + q\cdot C_{n-1,12}(q)\\ 
  & \hspace{3em} - (q{-}1)^2\cdot C_{n-1,21}(q)  + q(q{-}1)\cdot C_{n-1, 121}(q) \label{cn21}
\end{split}
\\
C_{n,121}(q) &= C_{n-1,2}(q) + (q{-}1)\cdot C_{n-1,21}(q)  \label{cn121}
\end{align}
\end{theorem}
\begin{proof}
  We have
  \begin{multline} \label{exp0}
    \rho( T_1T_2^{-1} )^n =  \rho( T_1T_2^{-1} )^{n-1} \cdot \rho (T_1T_2^{-1}) 
\\
     =   q^{-n}\bigl(C_{n-1,0}+ C_{n-1,1}\cdot T_1 + C_{n-1,2} \cdot T_2 + C_{n-1,12} \cdot T_1T_2  
               + C_{n-1,21} \cdot T_2T_1  + C_{n-1,121}\cdot T_1T_2T_1 \bigr)
                 \\  \cdot \bigl( -(q{-}1)\cdot T_1 + T_1T_2 \bigr)
\\
    =  q^{-n}\biggl( -(q{-}1)C_{n-1,0}\cdot T_1  -(q{-}1)C_{n-1,1}\cdot T_1^2   -(q{-}1)C_{n-1,2} \cdot T_2T_1 
\\
                -(q{-}1)C_{n-1,12} \cdot T_1T_2T_1   -(q{-}1)C_{n-1,21} \cdot T_2T_1^2   -(q{-}1)C_{n-1,121}\cdot T_1T_2T_1^2
\\
          + C_{n-1,0} \cdot T_1T_2+ C_{n-1,1}\cdot T_1^2T_2 + C_{n-1,2} \cdot T_2T_1T_2
\\
           + C_{n-1,12} \cdot T_1T_2T_1T_2   + C_{n-1,21} \cdot T_2T_1^2T_2  + C_{n-1,121}\cdot T_1T_2T_1^2T_2 \biggr)
\\ 
   = q^{-n}\biggl( \Bigl( -(q{-}1)C_{n-1,0}\cdot T_1  -(q{-}1)C_{n-1,2} \cdot T_2T_1  -(q{-}1)C_{n-1,12} \cdot T_1T_2T_1 +  C_{n-1,0} \cdot T_1T_2\Bigr)
\\
  + \Bigl\{-(q{-}1)C_{n-1,1}\cdot T_1^2 -(q{-}1)C_{n-1,21} \cdot T_2T_1^2 -(q{-}1)C_{n-1,121}\cdot T_1T_2T_1^2 +  C_{n-1,1}\cdot T_1^2T_2
\\
     +C_{n-1,2} \cdot T_2T_1T_2  + C_{n-1,12} \cdot T_1T_2T_1T_2   + C_{n-1,21} \cdot T_2T_1^2T_2  + C_{n-1,121}\cdot T_1T_2T_1^2T_2 \Bigr\} \biggr)
  \end{multline}
after collecting powers of $q$ and expanding.  In the last grouping, the first four terms inside the parentheses $( \ )$ 
involve only elements of the preferred basis;
the second eight terms in the pair of braces $\{\ \}$ all require further expansion, as follows. 
\begin{align}
 \begin{split}  -(q{-}1)C_{n-1,1}\cdot T_1^2 &= -(q{-}1)C_{n-1,1}\cdot ((q-1)T_1 + q) 
\\   &= -(q-1)^2C_{n-1,1}\cdot T_1 - q(q-1)C_{n-1,1} \end{split} \label{expa}
\\
 \begin{split}  -(q{-}1)C_{n-1,21} \cdot T_2T_1^2 &= -(q{-}1)C_{n-1,21} \cdot T_2((q-1)T_1 + q) 
\\    &= -(q-1)^2C_{n-1,21}\cdot T_2T_1 - q(q-1)C_{n-1,21} \cdot T_2  \end{split} \label{expb}
\\
\begin{split} -(q{-}1)C_{n-1,121}\cdot T_1T_2T_1^2 &= -(q{-}1)C_{n-1,121} \cdot T_1T_2((q-1)T_1 + q) 
\\    &   = -(q-1)^2C_{n-1,121}\cdot T_1T_2T_1 - q(q-1)C_{n-1,121} \cdot T_1T_2 \end{split} \label{expc}
\\
\begin{split} C_{n-1,1}\cdot T_1^2T_2 &= C_{n-1,1}\cdot( (q{-}1) T_1 + q) T_2 
\\    &= (q{-}1)C_{n-1,1}\cdot T_1T_2 + qC_{n-1,1} \cdot T_2
\end{split} \label{expd}
\\
 C_{n-1,2} \cdot T_2T_1T_2 &= C_{n-1,2} \cdot T_1T_2T_1 \label{expe}
\\
\begin{split}   
  C_{n-1,12} \cdot T_1T_2T_1T_2   &=  C_{n-1,12}\cdot T_1^2T_2T_1 = C_{n-1,12}((q{-}1)T_1 + q)T_2T_1
\\ &= (q{-}1)C_{n-1,12}\cdot T_1T_2T_1 + qC_{n-1,12}\cdot T_2T_1
\end{split} \label{expf}
\\
\begin{split}
  C_{n-1,21} \cdot T_2T_1^2T_2  &= C_{n-1,21}\cdot T_2 ((q{-}1)T_1 + q) T_2 = (q{-}1)C_{n-1,21} \cdot T_2T_1T_2 + qC_{n-1,21} \cdot T_2^2
\\ &= (q{-}1)C_{n-1,21} \cdot T_1T_2T_1 + qC_{n-1,21}\cdot ((q{-}1)T_2 + q) 
\\ &=  (q{-}1)C_{n-1,21} \cdot T_1T_2T_1 + q(q{-}1)C_{n-1,21}\cdot T_2 + q^2C_{n-1,21})
\end{split} \label{expg}
\\
\begin{split}
 C_{n-1,121}\cdot T_1T_2T_1^2T_2  &=  C_{n-1,121} \cdot T_1T_2((q{-}1)T_1 + q)T_2 
\\ &= (q{-}1)C_{n-1,121} \cdot T_1T_2T_1T_2 + qC_{n-1,121} \cdot T_1T_2^2
 \\ &= (q{-}1)C_{n-1,121} \cdot T_1^2T_2T_1 + qC_{n-1,121}\cdot T_1((q{-}1)T_2+ q)
\\ &= (q{-}1)C_{n-1,121} \cdot ((q{-}1)T_1+ q)T_2T_1 + qC_{n-1,121}\cdot T_1((q{-}1)T_2+ q)
\\ &= (q{-}1)^2C_{n-1,121}\cdot T_1T_2T_1 + q (q{-}1)C_{n-1,121}\cdot T_2T_1 
\\  & \hspace{3em} + q(q{-}1)C_{n-1,121}\cdot T_1T_2+ q^2C_{n-1,121} \cdot T_1
\end{split} \label{exph}
\end{align}
Collecting the constant terms from \eqref{expa} and \eqref{expg}, we get
\begin{align*}
  C_{n,0} &=  - q(q-1)C_{n-1,1} + q^2C_{n-1,21}.
\intertext{Collecting coefficients of $T_1$ from \eqref{exp0}, \eqref{expa}, \eqref{exph}, we get}
  C_{n,1} &=  -(q{-}1)C_{n-1,0} -(q-1)^2C_{n-1,1} + q^2C_{n-1,121}.
\intertext{Collecting coefficients of $T_2$ from \eqref{expb}, \eqref{expd}, and \eqref{expg}, we get}
  C_{n,2} &=  - q(q-1)C_{n-1,21}  +  qC_{n-1,1}  + q(q{-}1)C_{n-1,21} = qC_{n-1,1}.
\intertext{Collecting coefficients of $T_1T_2$ from \eqref{exp0}, \eqref{expc}, \eqref{expd}, and \eqref{exph}, we get}
  C_{n,12} &=  C_{n-1,0}  - q(q-1)C_{n-1,121} + (q{-}1)C_{n-1,1}  + q(q{-}1)C_{n-1,121} =  C_{n-1,0}  + (q{-}1)C_{n-1,1}
\intertext{Collecting coefficients of $T_2T_1$ from \eqref{exp0}, \eqref{expb}, \eqref{expf}, and \eqref{exph}, we get}
  C_{n,21} &=  -(q{-}1)C_{n-1,2} -(q-1)^2C_{n-1,21} + qC_{n-1,12}  + q (q{-}1)C_{n-1,121}.
\intertext{Collecting coefficients of $T_1T_2T_1$ from \eqref{exp0}, \eqref{expc}, \eqref{expe}, \eqref{expf}, \eqref{expg}, and \eqref{exph},
we get}
 C_{n,121} &=   -(q{-}1)C_{n-1,12}  -(q-1)^2C_{n-1,121} + C_{n-1,2}  \\ &+ (q{-}1)C_{n-1,12} +  (q{-}1)C_{n-1,21} + (q{-}1)^2C_{n-1,121}
\\
         &= C_{n-1,2} +  (q{-}1)C_{n-1,21}
\end{align*}
Up to simple rearrangements and expansion of notation, these are formulas \eqref{cn0} through ~\eqref{cn121}.
\end{proof}
\begin{example} \label{secondCs}
Applying the recursion formulas just proved to the table of initial polynomials, 
or by computing $\rho\bigl( (\sigma_1 \sigma_2^{-1})^2 \bigr)$ directly from the definitions, 
we find
\begin{align}
  C_{2,0}(q) &= q^2 \cdot C_{1,21}(q) - q(q{-}1)\cdot C_{1,1}(q) = q(q{-}1)^2,
\\
  C_{2,1}(q) &= -(q{-}1)^2\cdot C_{1,1}(q)-(q{-}1)\cdot C_{1,0}(q) = (q{-}1)^3,
\\
  C_{2,2}(q) &= q \cdot C_{1,1}(q) = -q(q{-}1),
\\
  C_{2,12}(q) &= (q{-}1)\cdot C_{1,1}(q) + C_{1,0}(q) = -(q{-}1)^2,
\\
  C_{2,21}(q) &= -(q{-}1)\cdot C_{1,2}(q) + q \cdot C_{1,12}(q) - (q{-}1)^2\cdot C_{1,21}(q) = q,
\\
  C_{2,121}(q) &=0.
\end{align}
  \end{example}
As a first application, we have the following vanishing result.
\begin{proposition}
  \label{C121}
For all $n$, $C_{n,121}(q) = 0$.
\end{proposition}
\begin{proof}
For $n \geq 1$, we claim $C_{n+1,121}(q) = 0$.
Make the inductive assumption that $C_{k,121}(q) = 0$ for $1 \leq k \leq n$. Apply \eqref{cn121}, \eqref{cn21}, and the 
inductive hypothesis to write
\begin{multline*}
       C_{n+1,121}(q) = C_{n,2}(q) + (q{-}1)\cdot C_{n,21}(q)
\\
  \shoveleft = C_{n,2}(q) 
\\
   + (q{-}1)\Bigl( -(q{-}1)\cdot  C_{n-1,2}(q) + q\cdot C_{n-1,12}(q) - (q{-}1)^2\cdot C_{n-1,21}(q)
 + q(q{-}1)\cdot C_{n-1, 121} \Bigr) 
\\
  = C_{n,2}(q) + 
    (q{-}1)\bigl( -(q{-}1)\cdot C_{n-1,2}(q) + q\cdot C_{n-1,12}(q) -
    (q{-}1)^2\cdot C_{n-1,21}(q) \bigr).
\end{multline*}
Using \eqref{cn2} to replace the first term $C_{n,2}(q)$ and \eqref{cn12} to replace the third term factor $C_{n-1,12}(q)$ on the right,
\begin{align*}
  C_{n+1, 121}(q)  &= q\cdot C_{n-1,1}(q) - (q{-}1)^2C_{n-1,2}(q) + q(q{-}1)\bigl( (q{-}1)C_{n-2,1}(q)+C_{n-2,0}(q) \bigr)
 \\
 &- (q{-}1)^3C_{n,21}(q) 
\\
  &= q\cdot C_{n-1,1}(q) - (q{-}1)^2C_{n-1,2}(q) + (q{-}1)^2\Bigl( q C_{n-2,1}(q) \Bigr) + q(q{-}1)C_{n-2,0}(q) 
\\ &- (q{-}1)^3C_{n,21}(q) 
\\
  &= q\cdot C_{n-1,1}(q) - (q{-}1)^2C_{n-1,2}(q) + (q{-}1)^2C_{n-1,2}(q)+ q(q{-}1)C_{n-2,0}(q)
\\ &- (q{-}1)^3C_{n,21}(q),
\end{align*}
and using \eqref{cn2} in reverse to rewrite the term $q C_{n-2,1}(q) $.  Making the obvious cancellation,
\begin{align*}
  C_{n+1, 121}  &=  q\cdot C_{n-1,1}(q)+ q(q{-}1) \cdot C_{n-2,0}(q) - (q{-}1)^3 \cdot C_{n,21}(q) 
\\
&=q\bigl( C_{n-1,1}+ (q{-}1)C_{n-2,0}\bigr) - (q{-}1)^3 \cdot C_{n-1, 21}
\\
&= q\Bigl(\bigl(-(q{-}1)^2 \cdot C_{n-2,1} - (q{-}1) \cdot C_{n-2,0}\bigr) + (q{-}1) \cdot C_{n-2,0}\biggr) - (q{-}1)^3 \cdot C_{n-1, 21},
\end{align*}
since 
\begin{align*}  
C_{n-1,1}(q) &= - (q{-}1)^2\cdot C_{n-2,1}(q) - (q{-}1)\cdot C_{n-2,0}(q) + q^2\cdot C_{n-2,121}(q)
\\
            & = - (q{-}1)^2\cdot C_{n-2,1}(q) - (q{-}1)\cdot C_{n-2,0}(q)
\end{align*}
by \eqref{cn1} and the inductive hypothesis.  Therefore, 
\begin{align*}
C_{n+1,121}(q)  &= -q(q{-}1)^2\cdot C_{n-2,1}(q)  - (q{-}1)^3\cdot C_{n-1, 21}(q)
\\
&= -(q{-}1)^2\cdot C_{n-1,2}(q) - (q{-}1)^3\cdot C_{n-1,21}(q),
\intertext{using \eqref{cn2} in the form $  C_{n-1,2}(q) =  q\cdot C_{n-2,1}(q) $,}
&= -(q{-}1)^2\bigl( C_{n-1, 2}(q) - (q{-}1)\cdot C_{n-1,21} (q)\bigr)
\\
&= -(q{-}1)^2\cdot C_{n, 121}(q) = 0,
\end{align*}
using \eqref{cn121} and the inductive hypothesis.
\end{proof}
\section{Obtaining the Jones Polynomial}   \label{JonesPoly}
Following the construction given in 
\cite[p.288]{Jones_poly86}
we work over the function field $K = \bC(q,z)$, and we put $w=1{-}q{+}z$.  Let $H_{N+1}$ be the Hecke algebra over $K$ corresponding
to $q$ with $N$ generators as in definition \ref{Heckealgebras}. The
starting point is the following theorem.
\begin{theorem}
  \label{traces}
For $N \geq 1$ there is a family of trace functions $\Tr \colon H_{N+1} \ra K$ compatible with the inclusions $H_N \ra H_{N+1}$ satisfying
\begin{enumerate}
 \item $\Tr(1) = 1$,
\item $\Tr$ is $K$-linear and $ \Tr(ab) = \Tr(ba)$,
\item If $a, b \in H_N$, then $\Tr (aT_Nb) = z\Tr(ab)$.
\end{enumerate}
\end{theorem}
Property 3 enables the calculation of $\Tr$ on basis elements of $H_{N+1}$ through use of the defining relations and induction. 
For $H_3$, note that
\begin{equation*}
  \Tr(T_1) = \Tr(T_2) = z, \quad \Tr(T_1T_2) = \Tr(T_2T_1) = z^2, \quad \Tr(T_1T_2T_1) = z \Tr(T_1^2) = z \bigl((q{-}1)z + q\bigr).
\end{equation*}
The next step toward the Jones polynomial of the knot that is the closure of
the braid $\alpha \in B_{N+1}$ is given by the formula
\begin{equation*}
  V_{\alpha}(q,z) = \Bigl(\frac{1}{z}\Bigr)^{(N + e(\alpha))/2}
                        \cdot \Bigl( \frac{q}{w} \Bigr)^{(N-e(\alpha))/2}\cdot \Tr\bigl(\rho(\alpha)\bigr),
\end{equation*}
where $e(\alpha)$ is the exponent sum of the word $\alpha$. The expression defines an element in the quadratic extension $K(\sqrt{q/zw})$. 
For the weaving knot $W(3,n)$, viewed as the closure of $(\sigma_1\sigma_2^{-1})^n$, we have the exponent sum $e=0$, and $N=2$, and 
\begin{equation*}
 \rho\bigl( (\sigma_1\sigma_2^{-1})^n \bigr)= (T_1T_2^{-1})^n 
        = q^{-n}\bigl( C_{n,0}(q) +  C_{n,1}(q)\cdot T_1 + C_{n,2}(q) \cdot T_2 + C_{n,12}(q) \cdot T_1T_2  + C_{n,21}(q) \cdot T_2T_1\bigr) ,
\end{equation*}
thanks to proposition \ref{C121}, which says the expression for $(T_1T_2^{-1})^n$ requires only the use of the basis elements
$1$, $T_1$, $T_2$,  $T_1T_2$ and $T_2T_1$. Then we have 
\begin{multline*}
  V_{(\sigma_1\sigma_2^{-1})^n}(q,z) \\  = \Bigl(\frac{1}{z}\Bigr)\cdot \Bigl( \frac{q}{w} \Bigr)\cdot q^{-n}
     \Tr \bigl(C_{n,0}(q) +  C_{n,1}(q)\cdot T_1 + C_{n,2}(q) \cdot T_2 + C_{n,12}(q) \cdot T_1T_2  
               + C_{n,21}(q) \cdot T_2T_1   \bigr)
\\
  = \Bigl(\frac{q}{zw}\Bigr)\cdot   q^{-n} \cdot \bigl( C_{n,0}(q) +   C_{n,1}(q) \cdot z + C_{n,2}(q) \cdot z + C_{n,12}(q) \cdot z^2
               + C_{n,21} (q) \cdot z^2 \bigr),
\end{multline*}
using the facts that $\Tr T_1 = \Tr T_2 = z$ and $\Tr T_1T_2 = \Tr T_2T_1 = z^2$. 
Consequently, the sums 
$ C_{n,1} + C_{n,2} $ and  $C_{n,12} + C_{n,21}$
are essential for understanding the two-variable Jones polynomial of $W(3,n)$, the closure of $\alpha=(\sigma_1\sigma_2^{-1})^n $. 
 Making the substitutions
\begin{equation*}
  q = t, \quad  z = \frac{t^2}{1+t}, \quad w = \frac{1}{1+t}
\end{equation*}
leads to the one-variable Jones polynomial
\begin{multline*}
  V_{W(3,n)}(t) = \frac{t(1{+}t)^2}{t^2} \cdot t^{-n} \cdot
     \Bigl( C_{n,0}(t) + (C_{n,1}(t)+C_{n,2}(t))\cdot \frac{t^2}{1{+}t} + (C_{n,12}(t) + C_{n,21}(t)) \cdot \frac{t^4}{(1{+}t)^2}\Bigr)
 \\
  = t^{-n-1}\cdot\bigl( (1{+}t)^2\cdot C_{n,0}(t) + (1{+}t)\cdot( C_{n,1}(t) + C_{n,2}(t) )\cdot t^2 + (C_{n,12}(t) + C_{n,21}(t))\cdot t^4 \bigr).
\end{multline*}
\begin{example}
  For $W(3,1)$, which is the unknot, we have
  \begin{align*}
 V_{W(3,1)}(t) 
   &=  t^{-2}\cdot\bigl( (1{+}t)^2\cdot C_{1,0}(t) + (1{+}t)\cdot( C_{1,1}(t) + C_{1,2}(t) )\cdot t^2 + (C_{1,12}(t) + C_{1,21}(t))\cdot t^4 \bigr) 
\\
  &=  t^{-2}\cdot\bigl( (1{+}t)^2\cdot 0 + (1{+}t)\cdot(-(t-1) + 0 )\cdot t^2 + (1 + 0 )\cdot t^4 \bigr)
\\
  &= t^{-2}\cdot ( (1{-}t^2) t^2 + t^4 ) = 1.
  \end{align*}
\end{example}
\begin{example} \label{jonesfig8knot}
  For $W(3,2)$, which is the figure-8 knot, we have
  \begin{align*}
    V_{W(3,2)}(t) 
&=t^{-3}\cdot \bigl( (1{+}t)^2\cdot C_{2,0}(t) + (1{+}t)\cdot( C_{2,1}(t) + C_{2,2}(t) )\cdot t^2 
                              + (C_{2,12}(t) + C_{2,21}(t))\cdot t^4 \bigr)
\\
&=t^{-3}\cdot \bigl( (1{+}t)^2\cdot t(t{-}1)^2 +(1{+}t)\cdot( (t{-}1)^3 - t(t{-}1)  ) \cdot t^2
                              +( -(t{-}1)^2+   t ) \cdot t^4 \bigr) 
\\
&= t^{-3}\cdot \bigl( t^5 - t^4 + t^3 -t^2 + t \bigr) = t^2 - t + 1 -t^{-1} + t^{-2}
  \end{align*}
\end{example}

Now we take a closer look at the formal expression 
\begin{multline*}
  V_{W(3,n)}(t) =
 \\
  = t^{-n-1}\cdot\bigl( (1{+}t)^2\cdot C_{n,0}(t) + (1{+}t)\cdot( C_{n,1}(t) + C_{n,2}(t) )\cdot t^2 + (C_{n,12}(t) + C_{n,21}(t))\cdot t^4 \bigr)
\end{multline*}
for the Jones polynomial of the weaving knot $W(3,n)$. 
\begin{proposition}
  \label{degrees1}
We have a uniform bound on the degrees of the polynomials $C_{n,*}$. Namely,
\begin{equation*}
  \deg(C_{n,*}) \leq 2n-1,
\end{equation*}
and the sharper bounds
\begin{equation}
  \deg(C_{n,2}) \leq 2n{-}2, \quad \deg(C_{n,12}) \leq 2n{-}2, \quad \text{and} \quad \deg(C_{n,21}) \leq 2n{-}3.
\end{equation}
\end{proposition}
\begin{proof}
  These are  easy arguments by induction, using either the formulas \eqref{initialCs} or the formulas in example \ref{secondCs}
to start the inductions.
Use the recursion formulas \eqref{cn0} through
  \eqref{cn21} along with the fact that $C_{n, 121}= 0$, proved in proposition \ref{C121}, for the inductive step.
We have 
\begin{align}
  \begin{split}
      \deg(C_{n,0}) &\leq \max\{ \deg(C_{n-1,21})+ 2, \; \deg(C_{n-1,1}) + 2\} 
  \\
               &\leq \max\{(2n{-}5)+2, (2n{-}3)+2 \} = 2n{-}1;
  \end{split}
\\
\begin{split}
  \deg(C_{n,1}) &\leq \max\{ \deg(C_{n-1,1}) + 2, \; \deg(C_{n-1,0}) + 1\} 
\\
               &\leq \max\{(2n{-}3)+2, (2n{-}3)+1\} = 2n-1;
\end{split}
\\
  \deg(C_{n,2}) &= \deg(C_{n-1,1}) + 1 \leq (2n{-}3) + 1 = 2n-2;
\\
\begin{split}
    \deg(C_{n,12}) &\leq \max\{ \deg(C_{n-1,1}) + 1, \deg(C_{n-1,0}) \}
\\
   &\leq \max \{ (2n{-}3)+1, 2n-3 \} = 2n-2;
\end{split}
\\
\begin{split}
  \deg(C_{n,21}) &\leq \max\{ \deg(C_{n-1,2}) + 1, \deg( C_{n-1,12}) +1, \deg(C_{n-1,21}) + 2 \} 
\\ 
               &\leq \max\{ (2n{-}4) + 1, (2n{-}4) + 1, (2n{-}5)+ 2\} = 2n{-}3.
\end{split}
  \end{align}
\end{proof}
 Accordingly, set
\begin{align*}
C_{n,0}(q) &= \sum_{i=0}^{2n-1} c_{n,0,i}q^i, &
  C_{n,1}(q) &= \sum_{i=0}^{2n-1} c_{n,1,i}q^i, & 
  C_{n,2}(q) &= \sum_{i=0}^{2n-2} c_{n,2,i}q^i,
\\
  C_{n,12}(q) &= \sum_{i=0}^{2n-2} c_{n,12,i}q^i,
& &\text{and} & 
 C_{n,21}(q) &= \sum_{i=0}^{2n-3} c_{n,21,i}q^i.
\end{align*}
\begin{lemma} \label{Cn0coefficients}
  In the polynomial $C_{n,0}(q)$, the constant term $c_{n,0,0} = 0$  for all   $n \geq 1$, 
  and the degree one coefficient $c_{n,0,1}=(-1)^{n-2}$ for $n \geq 2$.
\end{lemma}
\begin{proof}
  The first polynomial $C_{1,0}(q) = 0$, and setting $q=0$ in the recurrence relation \eqref{cn0} immediately yields
  \begin{equation*}
    c_{n,0,0} = C_{n,0}(0) = 0.
  \end{equation*}
Differentiate the recursion relation \eqref{cn0} with respect to $q$, obtaining
\begin{equation*}
  C_{n,0}'(q) = \bigl( 2q\cdot C_{n-1,21}(q)+ q^2\cdot C_{n-1,21}'(q) \bigr)  
              - \bigl( (2q{-}1)\cdot C_{n-1,1}(q)+ q(q{-}1) \cdot C_{n-1,1}'(q)\bigr) .
\end{equation*}
Substituting $q=0$ yields immediately $c_{n,0,1} = C_{n,0}'(0) = C_{n-1,1}(0) = c_{n-1,1,0}$.  We have 
\begin{equation*}
   C_{n,1}(q) = - (q{-}1)^2\cdot C_{n-1,1}(q) - (q{-}1)\cdot C_{n-1,0}(q), 
\end{equation*}
simplifying relation \eqref{cn1} using proposition \ref{C121} to set $C_{n-1,121}(q)=0$. 
Now we prove $c_{n,1,0} = (-1)^{n-1}$ for all $n \geq 1$. 
We have $C_{1,1}(q) = -(q-1)$, so $c_{1,1,0} = 1$ as claimed.  Substituting $q=0$ and using $c_{n,0,0}=0$ we get
\begin{equation*}
  c_{n,1,0} = C_{n,1}(0) = -(-1)^2\cdot C_{n-1,1}(0) - (-1) \cdot C_{n-1,0}(0) = -c_{n-1,1,0} + 0 = -(-1)^{n-2}=(-1)^{n-1}
\end{equation*}
Thus $c_{n,0,1} = c_{n-1,1,0} = (-1)^{n-2}$, as claimed.
\end{proof}
Thus, we improve the expression for $C_{n,0}(q)$ slightly, obtaining $C_{n,0}(q) = \sum_{i=1}^{2n-1} c_{n,0,i}q^i$. Now we examine
\begin{align*}
 V_{W(3,n)}(t) 
&= t^{-n-1}\cdot \bigl( (1{+}t)^2\cdot C_{n,0}(t) + (t^2{+}t^3)\cdot( C_{n,1}(t) + C_{n,2}(t) ) + t^4 \cdot (C_{n,12}(t) + C_{n,21}(t)) \bigr)
\\
\begin{split}
  &=  t^{-n-1}\cdot \Biggl( (1{+}t)^2\cdot \biggl(\sum_{i=1}^{2n-1} c_{n,0,i}t^i\biggr) 
\\
     & \hspace{0.15\linewidth} + (t^2{+}t^3)\cdot \Bigl( \sum_{i=0}^{2n-1} c_{n,1,i}t^i + \sum_{i=0}^{2n-2} c_{n,2,i}t^i \Bigr) 
\\
        & \hspace{0.30\linewidth} + t^4 \cdot \biggl(\sum_{i=0}^{2n-2} c_{n,12,i}t^i + \sum_{i=0}^{2n-3} c_{n,21,i}t^i\biggr) \Biggr).
\end{split}
\end{align*}
In the expression for $V_{W(3,n)}(t)$ the highest degree term is apparently
\begin{equation*}
  t^{-n-1}\cdot\bigl( c_{n,1,2n-1} \cdot t^{2n+2} + c_{n,12,2n-2} \cdot t^{2n+2} \bigr),
\end{equation*}
but we claim this term is actually zero.
\begin{proof}
  Using  $\deg(C_{n-1,1}) \leq 2n-3$, $\deg(C_{n-1,0}) \leq 2n-3$, the fact that $C_{n, 121}=0$, and simplifying the  recursion formulas
\eqref{cn1} and \eqref{cn12}, respectively, to 
\begin{equation*}
    C_{n,1}(q) = - (q{-}1)^2\cdot C_{n-1,1}(q) - (q{-}1)\cdot C_{n-1,0}(q)
\; \text{and} \;
   C_{n,12}(q) = (q{-}1)\cdot C_{n-1,1}(q) + C_{n-1,0}(q), 
\end{equation*}
 we compute
  \begin{equation*}
    c_{n,1,2n-1} + c_{n,12,2n-2} =  -c_{n-1,1,2n-3} + c_{n-1,1,2n-3} = 0. \qedhere
  \end{equation*}
\end{proof}
So the degree of the highest term in $V_{W(3,n)}$ is no more than the degree of $t^{-n-1}\cdot t^{2n+1}$ which is $n$.  
Lemma \ref{Cn0coefficients} shows that the term of lowest degree is $\pm t^{-n-1} \cdot t = \pm t^{-n}$, so the 
span of the Jones polynomial $V_{W(3,n)}$ is no more than $2n$. Of course, Kauffman \cite[Theorem~2.10]{States}
has proved that the span of the polynomial is, in fact, precisely $2n$. 
\section{From the Jones Polynomial to Khovanov homology}  \label{Jones-to-Khovanov}
In this section we amplify Theorem \ref{locateKH}, at least the first part of it. \addtocounter{section}{-3} \addtocounter{theorem}{4}
\begin{theorem}
For a weaving knot $W(2k{+}1,n)$ the non-vanishing Khovanov homology ${\mathcal H}^{i,j}\bigl( W(2k{+}1, n) \bigr)$ lies on the 
lines
\begin{equation*}
  j = 2i \pm 1.
\end{equation*}
For a weaving knot $W(2k, n)$ the non-vanishing Khovanov homology ${\mathcal H}^{i,j}\bigl( W(2k, n) \bigr)$ lies on the lines
\begin{equation*}
  j = 2i + n -1 \pm 1
\end{equation*}
 \end{theorem}
We have the following definition of the bi-graded Euler characteristic
associated to Khovanov homology. 
 \begin{equation*}
   Kh(L)(t,Q)  \stackrel{{\rm def}}{=} \sum t^iQ^j \dim {\mathcal H}^{i,j}(L)
 \end{equation*}
\addtocounter{section}{3} \addtocounter{theorem}{-5}
 \begin{theorem}[Theorem 1.1, \cite{Lee_Endo04}] \label{Lee1}
For an oriented link $L$, the graded Euler characteristic
\begin{equation*}
  \sum_{i,j \in \bZ} (-1)^iQ^j  \dim {\mathcal H}^{i,j}(L) 
\end{equation*}
   of the Khovanov invariant ${\mathcal H}(L)$ is equal to $(Q^{-1}{+}Q)$
   times the Jones polynomial $V_L(Q^2)$ of $L$.

In terms of the associated polynomial $Kh(L)$, 
\begin{equation} \label{JonesfromKhovanov}
  Kh(L)(-1, Q) = (Q^{-1}+Q)V_L(Q^2).
\end{equation}
 \end{theorem}
 \begin{theorem}[Compare Theorem 1.4 and subsequent remarks, \cite{Lee_Endo04}] \label{Lee4}
   For an alternating knot $L$, its Khovanov invariants ${\mathcal H}^{i,j}(L)$ of 
degree difference $(1,4)$ are paired except in the $0$th cohomology group. 
 \end{theorem}
This fact may be expressed in terms of the polynomial $Kh(L)$, as follows. 
There is another polynomial $Kh'(L)$ in one variable and an  equality 
 \begin{equation} \label{polynomialperiodicity}
   Kh(L)(t, Q) = Q^{-\sigma(L)}\bigl\{ (Q^{-1}{+}Q) + (Q^{-1}+tQ^2\cdot Q)\cdot Kh'(L)(tQ^2) \bigr\}
 \end{equation}
When we combine theorems \ref{Lee1} and \ref{Lee4}, we find that the bi-graded Euler characteristic and
the Jones polynomial of an alternating link determine one another.  Obviously, the equality \eqref{JonesfromKhovanov}
shows that one knows $V_L$ if one knows $Kh(t,Q)$. 

To obtain $Kh(t,Q)$ from $V_L(Q^2)$ requires a certain amount of manipulation.  Implementing these 
manipulations in {\em Maple} and {\em Mathematica} is an important step in our experiments.
 Setting $t{=}-1$ in \eqref{polynomialperiodicity} and combining with equation \eqref{JonesfromKhovanov}, one has
\begin{align*}
  (Q^{-1} + Q) \cdot V_L(Q^2) &= Q^{-\sigma(L)}\bigl\{ (Q^{-1}{+}Q) + (Q^{-1}- Q^3)\cdot Kh'(L)(-Q^2) \bigr\}.
\intertext{Consequently,}
 V_L(Q^2) &= Q^{-\sigma(L)}\bigl\{ 1 + \frac{(Q^{-1}- Q^3)}{(Q^{-1}{+}Q)}\cdot Kh'(L)(-Q^2) \bigr\}
\\
&= Q^{-\sigma(L)}\bigl\{ 1 + (1 - Q^2)\cdot Kh'(L)(-Q^2)\bigr\}.
\intertext{Furthermore,}
  Q^{\sigma(L)} \cdot V_L(Q^2) -1 &= (1 - Q^2)\cdot Kh'(L)(-Q^2),
\intertext{or}
Kh'(L)(-Q^2) &= (1 - Q^2)^{-1}\cdot \bigl(Q^{\sigma(L)} \cdot V_L(Q^2) -1\bigr) .
\end{align*}
Replacing $Q^2$ in the last equation by $-tQ^2$ is the last step to obtain $Kh'(L)$ from the Jones polynomial.
Within a computer algebra system, one must first replace $Q^2$ by $-X$ and then replace $X$ by $tQ^2$.  Once one has
$Kh'(L)(tQ^2)$, one obtains $Kh(t,Q)$ directly from equation \eqref{polynomialperiodicity}.
\begin{example}
  We have computed $V_{W(3,2)}(t) = t^{-2} - t^{-1} + 1 - t + t^2$ in example \ref{jonesfig8knot}, so
  \begin{align*}
    Kh'\bigl(W(3,2)\bigr)(-Q^2) &= (1- Q^2)^{-1} \cdot \bigl( Q^0 \cdot ( Q^{-4} -  Q^{-2} - Q^2 + Q^4) \bigr)
\\
  &= (1-Q^2)^{-1} \cdot \bigl( (1 - Q^2) \cdot ( Q^{-4} - Q^2) \bigr)
\\
&= Q^{-4} - Q^2.
  \end{align*}
It follows that $Kh'\bigl( W(3,2) \bigr) (tQ^2) = t^{-2}Q^{-4} + tQ^2$, and 
\begin{multline*}
  Kh\bigl(W(3,2)\bigr) (t, Q) = (Q+Q^{-1})+ (Q^{-1} + tQ^3)(t^{-2}Q^{-4} + tQ^2)
\\
 = t^{-2}Q^{-5} + t^{-1}Q^{-1} + Q^{-1} + Q + tQ + t^2Q^5.
\end{multline*}
\end{example}
\section{Khovanov homology examples}  \label{Khovanov}
Once one has the Khovanov polynomial one can make a plot of the Khovanov
homology in an $(i,j)$-plane as in this example. The Betti number $\dim KH^{i,j}\bigl( W(3,11) \bigr)$ is plotted
at the point with coordinates $(i,j)$.
\begin{figure}[h]
  \centering
  \includegraphics[width=4in]{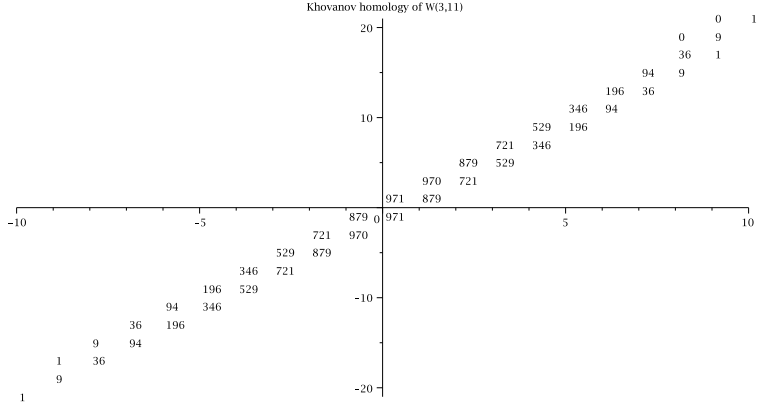}
  \caption{Khovanov homology of $W(3,11)$}
  \label{fig:w311}
\end{figure}
Clearly, as $n$ gets larger, it is going to be harder to make sense of such plots. Notice that the $(1,4)$-periodicity
of the Khovanov homology for these knots makes the information on one of the lines $j-2i = \pm 1$ redundant. 

Taking advantage of this feature, we simplify by recording the Betti numbers along the line $j- 2i = 1$. 
In order to study the asymptotic behavior of Khovanov homology we have to normalize the data.  This is done
by computing the total rank of the Khovanov homology along the line and dividing each Betti number by the
total rank.  We obtain normalized Betti numbers that sum to one.  

This raises the possibility of approximating the distribution of normalized Betti numbers by a probability distribution.
For our baseline experiments we choose to use the normal $N(\mu, \sigma^2)$ probability density function
\begin{equation*}
  f_{\mu, \sigma^2}(x) = \frac{1}{\sigma \sqrt{2\pi}} \exp \Bigl( - \frac{(x-\mu)^2}{2\sigma^2} \Bigr) 
\end{equation*}
Fit a quadratic function 
$q_n(x)= -(\alpha \, x^2 - \beta\, x + \delta)$ 
to the logarithms of the normalized Khovanov dimensions along the line $j=2i+1$
and exponentiate the quadratic function. 
Since the total of the normalized dimensions is 1, we normalize the
exponential, obtaining 
\begin{equation*}
  \rho_n(x) = A_n e^{q_n(x)} \quad \text{satisfying} \quad \int_{-\infty}^{\infty} \rho_n(x) \; dx = 1.
\end{equation*}
To obtain a formula for $A_n$, complete the square
\begin{equation*}
  q_n(x) = -\alpha \cdot \bigl( x - (\beta/2\alpha) \bigr)^2 +\bigl( (\beta^2/4\alpha) - \delta\bigr).
\end{equation*}
Then consider
\begin{align*}
  1 &= A_n \int_{-\infty}^{\infty} \exp q_n(x) \; dx
\\
    &= A_n \cdot 
 \int_{-\infty}^{\infty} \exp \bigl((\beta^2/4\alpha) - \delta \bigr) 
        \cdot \exp \bigl( -\alpha \cdot \bigl( x -(\beta/2\alpha) \bigr)^2\bigr)  \; dx
\\
   &= A_n \cdot \bigl((\beta^2/(4\alpha) - \delta \bigr)  \cdot
           \int_{-\infty}^{\infty} \exp \bigl( -\alpha \cdot \bigl( x -(\beta/2\alpha) \bigr)^2\bigr)  \; dx
\\
   &= A_n \cdot  \bigl((\beta^2/4\alpha)  - \delta\bigr)  \cdot  \sqrt{\pi/\alpha}
\end{align*}
Thus, the expression for $A_n$ is 
\begin{equation*}
  A_n =  \exp -\bigl((\beta^2/4\alpha) - \delta\bigr)\cdot \sqrt{\alpha/\pi}.
\end{equation*}

Equating the expressions
\begin{equation*}
  \rho_n(x) = \frac{1}{\sigma_n \sqrt{2\pi}} \exp \Bigl( - \frac{(x-\mu_n)^2}{2\sigma_n^2} \Bigr) 
\quad \text{and} \quad
\rho_n(x) = A_n \exp ( q_n(x))
\end{equation*}
$\mu_n = \beta/2\alpha$ and the efficient way to the parameter $\sigma_n$ is by solving the equation
\begin{equation*}
\frac{1}{\sigma_n \sqrt{2\pi}} = \rho_n( \frac{\beta}{2\alpha} ) = A_n \exp( q_n(\beta/2\alpha)) 
    = \exp -\bigl( \frac{\beta^2}{4\alpha} - \delta\bigr)\cdot \sqrt{\frac{\alpha}{\pi}} 
\exp\bigl( (\beta^2/4\alpha)  - \delta \bigr)
\end{equation*}
obtaining $\sigma_n = 1/ \sqrt{2\alpha} $.

Working this out for $W(3,10)$,  and carrying only 3 decimal places, 
the raw dimensions are
\begin{center}
  \begin{small}
     \begin{tabular}[h!]{c|c|c|c|c|c|c|c|c|c|c}
$i$  & -9 & -8 & -7 & -6 &  -5 &  -4 & -3&    -2 &-1 &  0
\\
$\dim$  &    1& 9& 36& 94& 196& 346& 529& 721& 879& 970
\\
 $i$  & 1& 2& 3& 4& 5& 6& 7& 8& 9& 10  
\\
$\dim$  & 971& 879& 721& 529& 346& 196& 94& 36& 9& 1 
  \end{tabular}
  \end{small}
\end{center}
and, to three significant digits, the logarithms of the normalized dimensions are
\begin{center}
  \begin{small}
    \begin{tabular}[h!]{c|c|c|c|c|c|c|c|c|c|c}
$i$  & -9 & -8 & -7 & -6 &  -5 &  -4 & -3&    -2 &-1 &  0
\\
    &  -17.9& -15.7& -14.3& -13.3& -12.6& -12.0& -11.6& -11.3& -11.1& -11.0
\\
 $i$  & 1& 2& 3& 4& 5& 6& 7& 8& 9& 10  
\\
& -11.0& -11.1& -11.3& -11.6& -12.0& -12.6& -13.3& -14.3& -15.7& -17.9                 
\end{tabular}
  \end{small}
\end{center}
Fitting a quadratic to this information,  we get
\begin{equation*}
          q_{10}(x) =  -10.7 + 0.0720\, x - 0.0720\, x^2,\quad \alpha = \beta =   0.0720, \quad \delta =     10.7.
\end{equation*}
To three significant digits $\mu_{10} = 0.500$ and $\sigma_{10} = 2.64$.

By the symmetry of Khovanov homology, the mean $\mu_n$ approaches $1/2$ rapidly, so this parameter is of little interest.
On the other hand, relating the parameter $\sigma_n$  to some geometric quantity, say, some hyperbolic invariant 
of the complement of the link, is a very interesting problem. 

For $W(3,10)$, the density function is
\begin{equation*}
  \rho_{10}(x) = 11686.8431618280538\,\sqrt {{\pi }^{-1}}
{{\rm e}^{- 10.7018780565714309+ 0.0716848579220777243\,x- 0.0716848579220778631
\,{x}^{2}}}
\end{equation*}
When placed into standard form, $\mu_{10} =0.5000054030 $ and $\sigma_{10} = 2.640882970 $.
Here is the comparison plot.
\begin{figure}[h]
  \centering
  \includegraphics[height=2.0in]{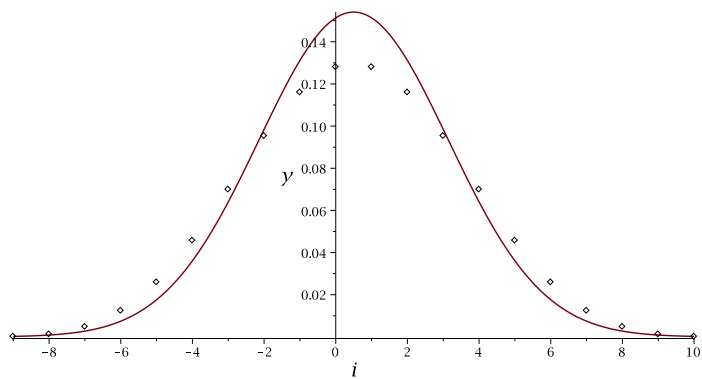}
  \caption{normalized homology of $W(3,10)$ compared with density function}
  \label{fig:w310comp}
\end{figure}
\newline
For the knot $W(3,11)$ 
 the expression for the density function is
\begin{equation*}
  \rho_{11}(x) =29676.8676257830375\,\sqrt {{\pi }^{-1}}{{\rm e}^{-
 11.6724860231789886+ 0.0661625395821569817\,x- 0.0661623073574252735
\,{x}^{2}}}
\end{equation*}
When placed into standard form, $\mu_{11} =0.5000017550 $ and $\sigma_{11} = 2.749031276$.
Here is the comparison plot.
\begin{figure}[h!]
  \centering
  \includegraphics[height=2.0in]{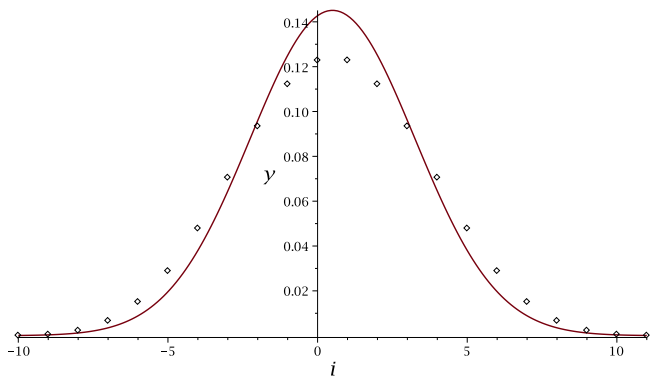}
  \caption{normalized homology of $W(3,11)$ compared with density function}
  \label{fig:w311comp}
\end{figure}
\newline
For $W(3,22)$, the density function is
\begin{equation*}
  \rho_{22}(x) =833596689.149608016\,\sqrt {{\pi }^{-1}}{{\rm e}^{-
 22.2219365040983057+ 0.0353061029354434300\,x- 0.0353061029347388616
\,{x}^{2}}}
\end{equation*}
When placed into standard form, $\mu_{22} =0.500000000 $ and $\sigma_{22} = 3.763224354$.
Here is the comparison plot.
\begin{figure}[h!]
  \centering
  \includegraphics[height=2.0in]{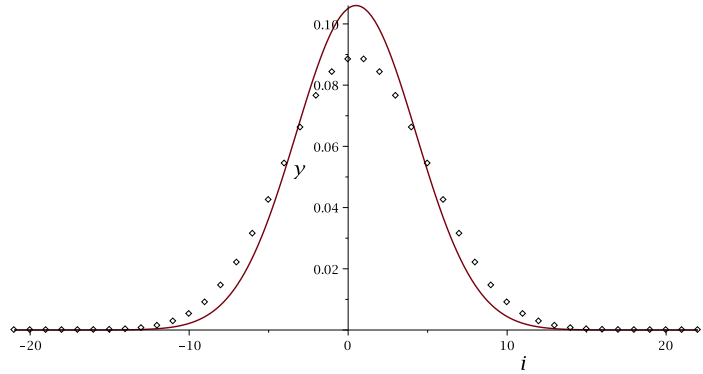}
  \caption{normalized homology of $W(3,22)$ compared with density function}
  \label{fig:w322comp}
\end{figure}
\newline
For $W(3,23)$, the density function is
\begin{multline*}
  \rho_{23}(x) = 2113964949.23002362\,\sqrt {{\pi }^{-1}} \\ 
\cdot {{\rm e}^{-
 23.1731352596503442+ 0.0338545815354610105\,x- 0.0338545815348441914
\,{x}^{2}}}
\end{multline*}
When placed into standard form, $\mu_{23} =0.5000000000 $ and $\sigma_{23} =3.843052143 $.
Here is the comparison plot.
\begin{figure}[h!]
  \centering
  \includegraphics[height=2.0in]{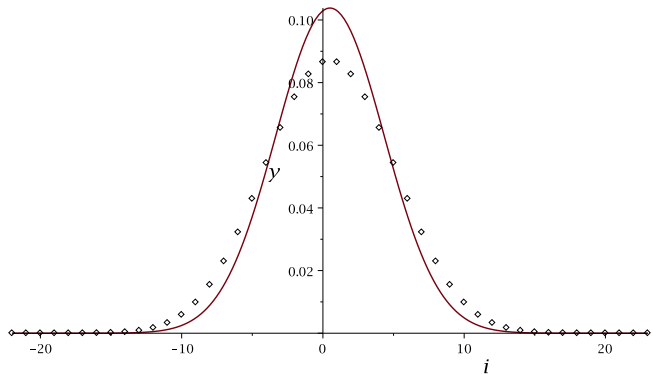}
  \caption{normalized homology of $W(3,23)$ compared with density function}
  \label{fig:w323comp}
\end{figure}
\newline
{\em Maple} worksheets and, later, {\em Mathematica} notebooks will be available at URL prepared by the second-named author.
\section{Data Tables}\label{Data}
This section contains tables of data generated using {\em Maple} to implement some of the results of earlier sections.
The first table collects data for weaving knots $W(3,n)$ for $n \equiv 1 \mod 3$. The first column lists the dimension; 
the second column lists the total dimension of the Khovanov homology lying along the line $j = 2i{+}1$; and the third
column lists the dimension of the vector space ${\mathcal H}^{0,1}\bigl( W(3,n)\bigr)$. 
In section \ref{Khovanov} we approximate the distribution of normalized
Khovanov dimenstions by a standard normal distribution, and we have displayed
graphics comparing the actual distribution with the approximation. 

To quantify those visual impressions, we compute two total deviations.  Let 
\begin{equation*}
  \text{Total dimension} = \sum_{i=-2n}^{2n+1} \dim {\mathcal H}^{i, 2i+1}\bigl( W(3,n) \bigr).
\end{equation*}
For the $L^2$-comparison, we compute
\begin{equation*}
 \Biggl(  \sum_{i = -2n}^{2n+1} \biggl( \rho_n(i) - \frac{\dim {\mathcal H}^{i, 2i+1}\bigl( W(3,n) \bigr)}{\text{Total dimension}} \biggr)^2 \Biggr)^{1/2}
\end{equation*}
For the $L^1$-comparison, we compute 
\begin{equation*}
   \sum_{i = -2n}^{2n+1} \Babs{ \rho_n(i) - \frac{\dim {\mathcal H}^{i, 2i+1}\bigl( W(3,n) \bigr)}{\text{Total dimension} } }
\end{equation*}
The $L^2$ comparisons appear to tend to 0, whereas the $L^1$ comparisons appear to be growing slowly.  
\begin{table} \caption{Data for $W(3,n)$ with $n \equiv 1 \mod 3$}
  \begin{tabular}[h!]{|c|c|c|c|c|c|}
    $n$ & Total dimension & $\dim {\mathcal H}^{0,1}$  & $\sigma$  &  $L^2$-comparison & $L^1$ comparison
\\
    10 & 7563 &   970 & 2.64088  &   0.040509 &  0.134828
\\
    13 & 135721 & 15418 &  2.95616 & 0.0411329 & 0.150599
\\
    16 & 2435423 &  250828 & 3.24564 &  0.040792 & 0.155995
\\
    19  & 43701901 & 4146351 & 3.51395 & 0.040145 & 0.161336
\\
    22  & 784198803 & 69337015 &  3.76322 &  0.039413 & 0.165763
\\
    25  & 14071876561 &  1169613435 & 3.99810 & 0.038678 & 0.167576
\\
    28  &  252509579303 & 19864129051 & 4.22032  & .037971 & 0.167790
\\
    31  &  4531100550901 & 339205938364 & 4.43167 & 0.0373026 & 0.170736
\\
    34  &  81307300336923 & 5818326037345 & 4.63358 & 0.0366758 & 0.172391
\\
    37  & 1459000305513721 &  100173472277125 & 4.82718 & 0.036089 & 0.173119
\\
    40  & 26180698198910063  & 1730135731194046  & 5.01342 & 0.035541 & 0.173178
\\ 
    43  & 469793567274867421 &  29963026081609060  &  5.19305 &  0.035027 & 0.173811
\\
    46  & 8430103512748703523 & 520131503664409798 & 5.36671 & 0.034547 & 0.175059
\\
    49  & $1.51272\cdot 10^{20}$ & $ 9.04765\cdot10^{18}$  &  5.53502 & 0.0340935 & 0.175779
\\
    52  &  $ 2.71447\cdot 10^{21}$  &  $ 1.57670\cdot 10^{20} $ & 5.69838 & 0.033667 & 0.176100
\\
    55  &  $ 4.87091\cdot 10^{22} $ & $ 2.75210\cdot 10^{21} $ & 5.85721 & 0.033265 & 0.176098
\\
    58  &  $ 8.74050\cdot 10^{23} $ & $ 4.81071\cdot 10^{22} $   & 6.01187 & 0.032885 & 0.175898
\\
    61  &  $ 1.56842\cdot 10^{25} $& $  8.42017\cdot 10^{23}$ & 6.16267 & 0.032524 & 0.176778
\\
    64  &  $ 2.81441\cdot 10^{26} $ & $ 1.47552\cdot 10^{25} $ & 6.30989 & 0.032182 & 0.177369
\\
    67  &  $5.05026\cdot 10^{27}$  & $ 2.58843\cdot 10^{26} $ & 6.45376 & 0.031857 & 0.177716
\\
    70  &$ 9.06233\cdot 10^{28} $& $4.54520\cdot 10^{27}$ & 6.59451 & 0.031547 & 0.177859
\\
    73  & $1.62617\cdot 10^{30}$ & $7.98842\cdot 10^{28}$ & 6.73233 & 0.031251 & 0.177831
\\
    76  & $ 2.91804\cdot 10^{31}$ & $1.40517 \cdot 10^{30}$ & 6.86740& 0.030968& 0.177657
\\
    79  &  $5.23621\cdot 10^{32}$ & $ 2.47359 \cdot 10^{31}$ & 6.99986& 0.030697& 0.177995
\\
    82  & $9.39600\cdot 10^{33} $&  $4.35747 \cdot 10^{32}$& 7.12988& 0.030437& 0.178445
\\
    85  &  $1.68604\cdot 10^{35} $&  $7.68116 \cdot 10^{33}$ & 7.25757& 0.030188& 0.178746
\\
    88  & $ 3.02548\cdot 10^{36} $ &  $1.35483 \cdot 10^{35}$ & 7.38305& 0.029948& 0.178918
\\
    91  &  $ 5.42901\cdot 10^{37} $ &  $  2.39106 \cdot 10^{36} $ & 7.50645& 0.029718& 0.178976
\\
   94  & $ 9.74196\cdot 10^{38}$ & $ 4.22211 \cdot 10^{37} $ & 7.62786& 0.029496& 0.178935
\\
    97 &  $ 1.74812\cdot 10^{40} $ & $7.45910  \cdot 10^{38}$ & 7.74736& 0.029282& 0.178807
\\
100  & $3.13688\cdot 10^{41} $ &  $1.31840   \cdot 10^{40} $& 7.86506& 0.029075& 0.178890
\\
121  &  $1.87923\cdot 10^{50} $&   $ 7.18477\cdot 10^{48} $ & 8.64424& 0.027805& 0.179577
\\
142  &  $1.12580\cdot 10^{59} $ & $ 3.97500 \cdot 10^{57} $ & 9.35886& 0.026769& 0.180247
\\
163  & $6.74436\cdot 10^{67} $   &  $2.22337 \cdot 10^{66}$ & 10.0227& 0.025900& 0.180596
\\
184  & $4.04037\cdot 10^{76}$ & $ 1.25398 \cdot 10^{75} $ & 10.6453& 0.025156& 0.180629
\\
205  & $ 2.42049\cdot 10^{85}$ &  $7.11854 \cdot 10^{83} $ & 11.2334& 0.024508& 0.180907
\\
247 & $8.68689\cdot 10^{102}$ & $2.32816 \cdot 10^{101} $ & 12.3258& 0.023423& 0.181027
\\
289 & $3.11764\cdot 10^{120} $ & $7.72623 \cdot 10^{118} $ & 13.3289& 0.022542& 0.181268
  \end{tabular}
\end{table}
\begin{table} \caption{Data for $W(3,n)$ with $n \equiv 2 \mod 3$}
\begin{tabular}{c|c|c|c|c|c}
    $n$ & Total dimension & $\dim {\mathcal H}^{0,1}$  & $\sigma$  &    $L^2$-comparison & $L^1$ comparison
\\
    11 & 19801&2431& 2.74903& 0.040906&  0.141925
\\
  14&355323&38983& 3.05533& 0.041079&  0.153170
\\
 17&6376021&637993& 3.33710& 0.040595&
 0.156595
\\
20&114413063&10591254& 3.59850&
 0.039905& 0.163190
\\
23&2053059121&177671734& 3.84305&
 0.039166& 0.166596
\\
26&36840651123&3004390818& 4.07348&
 0.038438& 0.167789
\\
29&661078661101&51124396786& 4.29190&
 0.037744& 0.168941
\\
 32&11862575248703&874400336044& 4.49997
& 0.037089& 0.171411
\\
35&212865275815561&15018149469823&
 4.69899& 0.036476& 0.172723
\\
38&3819712389431403&258853011125599&
 4.89004& 0.035903& 0.173203
\\
41&68541957733949701&4474997964407374&
 5.07400& 0.035366& 0.173083
\\
44&1229935526821663223&
77563025486587315& 5.25158& 0.034864& 0.174290
\\
47&22070297525055988321&
1347390412214087833& 5.42341& 0.034392& 0.175346
\\
50& $ 3.96035\cdot 10^{20} $& $ 2.34525
\cdot 10^{19} $& 5.59000& 0.033949& 0.175926
\\
53& $ 7.10657\cdot 10^{21} $& $ 4.08927
\cdot 10^{20} $& 5.75181& 0.033531& 0.176131
\\
56& $ 1.27522\cdot 10^{23} $& $ 7.14133
\cdot 10^{21} $& 5.90921& 0.033136& 0.176037
\\
59& $ 2.28829\cdot 10^{24} $& $ 1.24888
\cdot 10^{23} $& 6.06255& 0.032763& 0.176227
\\
62& $ 4.10617\cdot 10^{25} $& $ 2.18679
\cdot 10^{24} $& 6.21213& 0.032408& 0.177005
\\
65& $ 7.36823\cdot 10^{26} $& $ 3.83347
\cdot 10^{25} $& 6.35821& 0.032072& 0.177510
\\
68& $ 1.32218\cdot 10^{28} $& $ 6.72713
\cdot 10^{26} $& 6.50102& 0.031752& 0.177785
\\
71& $ 2.37255\cdot 10^{29} $& $ 1.18163
\cdot 10^{28} $& 6.64077& 0.031446& 0.177867
\\
74& $ 4.25736\cdot 10^{30} $& $ 2.07736
\cdot 10^{29} $& 6.77765& 0.031155& 0.177787
\\
77& $ 7.63953\cdot 10^{31} $& $ 3.65504
\cdot 10^{30} $& 6.91183& 0.030876& 0.177602
\\
80& $ 1.37086\cdot 10^{33} $& $ 6.43571
\cdot 10^{31} $& 7.04347& 0.030609& 0.178163
\\
83& $ 2.45990\cdot 10^{34} $& $ 1.13397
\cdot 10^{33} $& 7.17269& 0.030353& 0.178561
\\
86& $ 4.41412\cdot 10^{35} $& $ 1.99933
\cdot 10^{34} $& 7.29963& 0.030107& 0.178817
\\
89& $ 7.92082\cdot 10^{36} $& $ 3.52717
\cdot 10^{35} $& 7.42441& 0.029871& 0.178949
\\
92& $ 1.42133\cdot 10^{38} $& $ 6.22605
\cdot 10^{36} $& 7.54714& 0.029643& 0.178972
\\
95& $ 2.55048\cdot 10^{39} $& $ 1.09958
\cdot 10^{38} $& 7.66790& 0.029424& 0.178901
\\
98& $ 4.57665\cdot 10^{40} $& $ 1.94290
\cdot 10^{39} $& 7.78679& 0.029212& 0.178747
\\
 119& $ 2.74175\cdot 10^{49} $& $ 1.05696
\cdot 10^{48} $& 8.57308& 0.027914& 0.179650
\\
140& $ 1.64251\cdot 10^{58} $& $ 5.84051
\cdot 10^{56} $& 9.29316& 0.026859& 0.180257
\\
161& $ 9.83989\cdot 10^{66} $& $ 3.26385
\cdot 10^{65} $& 9.96138& 0.025977& 0.180552
\\
182& $ 5.89483\cdot 10^{75} $& $ 1.83951
\cdot 10^{74} $& 10.5875& 0.025223& 0.180539
\\
203& $ 3.53144\cdot 10^{84} $& $ 1.04367
\cdot 10^{83} $& 11.1787& 0.024566& 0.180926
\\
245& $ 1.26740\cdot 10^{102} $& $ 3.41053
\cdot 10^{100} $& 12.2759& 0.023469& 0.181064
\\
287& $ 4.54858\cdot 10^{119} $& $ 1.13115
\cdot 10^{118} $& 13.2829& 0.022580& 0.181221
\\
329& $ 1.63244\cdot 10^{137} $& $ 3.79224
\cdot 10^{135} $& 14.2187& 0.021838& 0.181399
\end{tabular}  
\end{table}
\newpage
\bibliographystyle{plain}
\bibliography{knotinvariantslist}

\begin{thebibliography}{1}

\bibitem{Weaving_vol}
Abhijit Champanerkar, Ilya Kofman, and Jessica~S. Purcell.
\newblock Volume bounds for weaving knots.
\newblock {\em Algebr. Geom. Topol.}, 16(6):3301--3323, 2016.

\bibitem{Jones_poly86}
Pierre de~la Harpe, Michel Kervaire, and Claude Weber.
\newblock On the {J}ones polynomial.
\newblock {\em Enseign. Math. (2)}, 32(3-4):271--335, 1986.

\bibitem{States}
Louis~H. Kauffman.
\newblock State models and the {J}ones polynomial.
\newblock {\em Topology}, 26(3):395--407, 1987.

\bibitem{Lee_Endo04}
Eun~Soo Lee.
\newblock An endomorphism of the {K}hovanov invariant.
\newblock {\em Adv. Math.}, 197(2):554--586, 2005.

\end{thebibliography}
\end{document}